\documentclass[12pt]{amsart}
\usepackage[latin1]{inputenc}
\usepackage{latexsym, amsfonts, amsmath, 
amscd,bezier}
\usepackage[brazil,portuguese,english]{babel}
\usepackage{mathrsfs}    
\usepackage{mathptmx} 
\usepackage{mathtools}  
\usepackage{amsmath,amssymb,amsthm,amsfonts,mathtools}
\usepackage{esint}
\usepackage[usenames,dvipsnames,svgnames,table]{xcolor}
\def\Qed{\ifhmode\unskip\nobreak\fi\quad 
  \ifmmode\square\else$\square$\fi} 
\usepackage[lmargin=2.5cm,tmargin=2cm,bmargin=2cm,rmargin=2.5cm]{geometry}

\usepackage{indentfirst}

\usepackage[bookmarksnumbered,plainpages]{hyperref}

\usepackage{soul} 

\newtheorem{theorem}{Theorem}[section]

\newtheorem{lemma}{Lemma}[section]

\numberwithin{equation}{section}

\newcounter{remark}[section]
\newenvironment{remark}
{\refstepcounter{remark}\medskip\noindent{\sc Remark\ \thesection.\theremark:}}{\medskip}

\newcounter{alphatheo}
\renewcommand\thealphatheo{\Alph{alphatheo}}
\newenvironment{alphatheo}{\refstepcounter{alphatheo}\medskip\noindent{\bf Theorem \thealphatheo.} \it }{\medskip}

\newcounter{example}

\renewenvironment{proof}{\medskip\noindent{\sc Proof:}}{\medskip}

\newcommand{\hel}
{
\hskip2.5pt{\vrule height8pt width.5pt depth0pt}
\hskip-.2pt\vbox{\hrule height.5pt width8pt depth0pt}
\,
}

\newcommand{\R}{\mathbb R}

\newcommand{\N}{\mathbb N}

\def\calL{{\mathcal L}}

\def\<{\langle}
\def\>{\rangle}

\def\calD{{\mathcal D}}

\def\calH{{\mathcal H}}

\DeclareMathOperator*{\diver}{div}

\begin{document}

\title[Hausdorff dimension of removable sets
\\for elliptic and canceling operators]{Hausdorff dimension of removable sets
\\for elliptic and canceling homogeneous differential operators\\
in the class of bounded functions}

\author {V. Biliatto}
\address{Departamento de Matem\'atica, Universidade Federal de S\~ao Carlos, S\~ao Carlos, SP,
13565-905, Brazil}
\email{{vbiliatto@ufscar.br}}

\author {L. Moonens}
\address{
Laboratoire de Math\'ematiques d'Orsay UMR~8628, B\^atiment 307 (IMO), Universit\'e Paris-Saclay,  F-91405 Orsay, France}
\email{Laurent.Moonens@universite-paris-saclay.fr}

\author {T. Picon}
\address{Departamento de Computa\c{c}\~ao e Matem\'atica, Universidade de S\~ao Paulo, Ribeir\~ao Preto, SP, 14040-901, Brazil}
\email{picon@ffclrp.usp.br}

\thanks{The first author was supported by Coordena\c{c}\~ao de Aperfei\c{c}oamento de Pessoal de N\'ivel Superior (CAPES - grant 88882.441243/2019-01) and the third by Conselho Nacional de Desenvolvimento Cient\'ifico e Tecnol\'ogico (CNPq - grant 311430/2018-0) and Funda\c{c}\~ao de Amparo \`a  Pesquisa do Estado de S\~ao Paulo (FAPESP - grant 18/15484-7). All three authors also acknowledge the support of the Brazilian-French Network in Mathematics (GDRI-RFBM).}



\begin{abstract}
In this note we give an upper bound on the Hausdorff dimension of removable sets
for elliptic and canceling homogeneous differential operators with constant coefficients 
in the class of bounded functions, using a simple extension of Frostman's lemma in Euclidean space with an additional power decay.
\end{abstract}

\maketitle



\section{Introduction}

Given a linear differential operator $P(x,D)$ with smooth coefficients in $\R^n$, $n\geq 2$, one calls a closed set $S\subseteq\R^n$ removable for the equation $P(x,D)f=0$ with respect to a space $\mathcal{F}$ of locally integrable functions (scalar or vector-valued), provided that for any $f\in \mathcal{F}$ satisfying (in the sense of distributions) the equation $P(x,D)f=0$ outside $S$, one has $P(x,D)f=0$ in $\R^n$ (in the sense of distributions). The nomenclature ``removable'' is hence self-explanatory.

The following result dates back to Harvey and Polking \cite[Theorem~4.1, (b)]{HP}, where $\calH^s$ will stand for the $s$-dimensional Hausdorff (outer) measure in $\R^n$.
\begin{theorem}\label{thm.HP}
If $P(x,D)$ is a linear differential operator of order $m<n$ with smooth coefficients and if the closed set $S\subseteq\R^n$ satisfies $\calH^{n-m}(S)=0$, then $S$ is removable for the equation $P(x,D)f=0$ with respect to the space $L^\infty_{loc}(\R^n)$ of locally (essentially) bounded functions.
\end{theorem}

Removable sets for several linear equations have been studied, and sometimes characterized completely, in the literature. For the scalar Laplace equation, removable sets with respect to Lipschitz functions have been subject of a very deep study; for example, it follows from works by Calder\'on \cite{CALDERON}, David and Mattila ($n=2$) \cite{DM} and Nazarov, Tolsa and Volberg ($n>2$) \cite{NTV} that a compact set $S\subseteq\R^n$ satisfying $\calH^{n-1}(S)<+\infty$ is removable for the Laplace equation with respect to Lipschitz functions, if and only if it is purely $(n-1)$-unrectifiable, \emph{i.e.} if and only if the intersection of $S$ with any $(n-1)$-rectifiable set, is $\calH^{n-1}$-negligible; (un)rectifiability hence plays a fundamental role for determining whether or not a set is removable in this context. 

The situation is very different for the divergence equation with respect to bounded vector fields, even though the Laplace equation can be written $\Delta f =\diver(\nabla f)=0$ and $\nabla f$ is a bounded vector field for any Lipschitz function $f$ and hence every removable set for the divergence equation with respect to bounded vector fields is removable for the Laplace equation w.r.t. Lipschitz functions.

It was first proven by the second author in \cite{M2006} that a compact set $S\subseteq\R^n$ is removable for the equation $\diver f=0$ with respect to $L^\infty(\R^n,\R^n)$, if and only if one has $\calH^{n-1}(S)=0$. The proof there used a decomposition of $S$ into Borel subsets $S_1$ and $S_2$, one of which is $(n-1)$-rectifiable, the other one being purely $(n-1)$-unrectifiable, relying then on results by Th.~De Pauw \cite{DP} for purely $(n-1)$-unrectifiable sets, and on the fact that $(n-1)$-rectifiable sets of positive $(n-1)$-dimensional Hausdorff measure are not removable for the Laplace equation w.r.t. Lipschitz functions (and hence also not removable for the divergence equation w.r.t. bounded vector fields, as discussed just above). Obviously, such a proof heavily relies on the fact that one deals with the \emph{divergence} operator, and cannot be carried out to other differential operators (even of order one).

Shortly after, Phuc and Torres \cite{PT} obtained, among other results, a new proof of the abovementioned characterization of (compact) removable sets for the divergence equation w.r.t. bounded vector fields, this time relying on a new strategy to prove that a compact set $S\subseteq\R^n$ with $\calH^{n-1}(S)>0$, cannot be removable for the divergence equation. Exhibiting first, with use of Frostman's lemma \cite[Theorem~8.8]{MATTILA}, a non-trivial Radon measure $\mu$ supported in such a set $S$ enjoying that $\mu(B[x,r])\leq r^s$ for all $x\in\R^n$ and $r>0$, and using a boxing inequality together with a co-area formula, they obtain an inequality of the form:
$$
\left|\int_{\R^n} \varphi\,d\mu \right|\leq C\|\nabla\varphi\|_{L^1},
$$
for smooth test functions $\varphi$, implying in turn that $\mu$ is in the dual space $X^*$ of the space $X$ of test functions endowed with the norm $\|\nabla\varphi\|_{L^1}$. Then, since the operator
$
-\nabla : X\to L^1(\R^n,\R^n)
$
is (obviously) an isometry, it follows that the adjoint operator:
$$
\diver=(-\nabla)^*: L^\infty(\R^n,\R^n)=L^1(\R^n,\R^n)^*\to X^*
$$
is surjective, and hence that the equation $\diver f=\mu$ has a solution in $L^\infty(\R^n,\R^n)$. As a consequence, $S$ cannot be removable for the divergence equation w.r.t. $L^\infty$ since one has $\diver v=0$ outside $S$ (recall that $\mu$ is supported in $S$) but $\diver f=\mu\neq 0$ in $\R^n$. The very specific role of the gradient \emph{vs} divergence, arising through the co-area formula in the latter argument, suggests that it does not adapt easily to obtain removability results for other operators than the divergence operator.

Very recently, the first and third authors obtained in \cite{BP}, for a special class of elliptic homogeneous differential operators $A(D):{\calD(\R^n,E)}\rightarrow {\calD(\R^n,F)}$ of order $0<m<n$ in $\R^n$ (see the Section \ref{sec.haus-dim} for details), sufficient conditions on the Radon measure $\mu$ in order to obtain solvability results in Lebesgue spaces for the equation:
\begin{equation}\label{ad}
A^{*}(D)f=\mu,
\end{equation}
where $A^{*}(D)$ is the (formal) adjoint operator associated to the homogeneous {linear} differential operator $A(D)$. 
In particular, the following solvability result in $L^{\infty}$ for equation \eqref{ad} was obtained as \cite[Theorem~B, p.~3]{BP}.
\begin{theorem}\label{thm.BP}
Assume that $A(D)$ is a homogeneous linear differential operator on $\R^n$ of order $0<m<n$, from a finite-dimensional vector space $E$ to a finite-dimensional vector space $F$ and that $\mu$ is a (vector-valued) Radon measure in $\R^n$ with values in $E^*$. If $A(D)$ is elliptic and canceling, and if one has:
\begin{equation}\label{eq.BP-1}
\sup_{r>0}\frac{|\mu|(B[0,r])}{r^{n-m}}<+\infty,
\end{equation}
as well as, uniformly on $x\in\R^n$:
\begin{equation}\label{eq.BP-2}
\int_0^{\frac{|x|}{2}} \frac{|\mu|(B[x,r])}{r^{n-m+1}}\,dr\lesssim 1,
\end{equation}
then there exists $f\in L^\infty(\R^n,F^*)$ satisfying the equation \eqref{ad} in $\R^n$ in the sense of distributions.
\end{theorem}
Here $|\mu|$ denotes the total variation of the vector-valued Radon measure $\mu$. {Note also that, in the latter statement:
\begin{itemize}
    \item The assumption \eqref{eq.BP-1} is weaker than requiring ${|\mu|(B[x,r])}\leq C {r^{n-m}}$ for all $x\in\R^n$ and all $r>0$, since the supremum only extends, in \eqref{eq.BP-1}, to balls \emph{centered at the origin}.
    \item Any scalar Radon measure which satisfies, for each $x\in\R^n$ and each $0<r<\frac{|x|}{2}$:
\begin{equation} \label{mu x regularity}
	\nu(B[x,r]) \leq C_2 \, |x|^{-m} r^n,
\end{equation}
automatically satisfies \eqref{eq.BP-2}; hence \eqref{mu x regularity} is a stronger condition than \eqref{eq.BP-2}.
\item The integration boundary $|x|/2$ in \eqref{eq.BP-2} can be replaced by $a|x|$, where $0<a<1$ is any fixed constant~---~the holding of \eqref{mu x regularity} for any $x\in\R^n$ and $0<r < a|x|$ then being again stronger that the modified version of \eqref{eq.BP-2}.
\item An example of measure $\nu$ satisfying \eqref{eq.BP-1} and \eqref{eq.BP-2} is given by 
$\nu = |x|^{{-m}} \calL^n$ for $n\geq 2$, where $\calL^n$ denotes Lebesgue's outer measure in $\R^n$.
\item The canceling property appearing in the statement (and defined precisely below in \eqref{canceling}) is due to J. Van Schaftingen (see  \cite{VS}); it characterizes the validity of an $L^{1}$ Sobolev-Gagliardo-Nirenberg inequality for elliptic homogeneous differential operators, recovering several \emph{a priori} $L^{1}$ estimates for divergence-vector fields and chains of complexes.
\end{itemize}}

In this note we present a necessary condition for a compact set $S\subseteq\R^n$ to be removable for the equation $A^*(D)f=0$ associated to an elliptic and canceling homogeneous differential operator $A(D)$, using Theorem \ref{thm.BP} and a slightly improved version of Frostman's lemma.
Our main result is the following:

\begin{alphatheo}\label{thm.main}
Assume that $A(D)$ is an elliptic and canceling homogeneous differential operator on $\R^n$ of order $0<m<n$, from a finite-dimensional vector space $E$ to a finite-dimensional vector space $F$.
If the closed set $S\subseteq\R^n$ is removable for the equation $A^*(D) f=0$ in $L^\infty(\R^n,F^*)$, then $S$ has Hausdorff dimension less than or equal to $n-m$.
\end{alphatheo}


\begin{remark}
Since it follows from Harvey and Polking's result (see Theorem~\ref{thm.HP}) that if $S$ is $\calH^{n-m}$-negligible, then $S$ is removable for the equation $A^*(D)f=0$ w.r.t. $L^\infty$, it hence only remains to understand whether or not some sets with Hausdorff dimension $n-m$ yet positive $(n-m)$-dimensional Hausdorff measure, may be removable in this context.
\end{remark}

We shall organize the paper as follows. In Section~\ref{sec.Frostman}, we shall present a version of Frostman's Lemma with an additional power decay condition. In Section~\ref{sec.haus-dim}, we shall then recall precisely the context of elliptic and canceling operators, before proving our main Theorem~\ref{thm.main}.

\section{A ``Frostman's lemma'' with decay}\label{sec.Frostman}

{
Our goal in this section is to provide a result ensuring at least that, given integers $0<m<n$ and a closed set $S\subseteq \R^n$ satisfying $\calH^{n-m+\alpha}(S)>0$ for some $\alpha>0$, there exists a (nonnegative) non-trivial Radon measure supported in $S$ and satisfying conditions \eqref{eq.BP-1} and \eqref{eq.BP-2} above. This will result from observing that one can impose, in the statement of Frostman's Lemma, a decay condition; this is what the next result expresses.
}

\begin{lemma}[Frostman's Lemma with power weight decay]\label{lem.frost}
Assume that $0<\alpha<s<n$ are fixed and that $B\subseteq\R^n$ is a Borel set satisfying $\calH^s(B)>0$. Then there exists a non-zero scalar Radon measure $\mu$ supported in $B$ satisfying:
\begin{equation}\label{eq.cond-1}
\sup_{r>0} \frac{\mu(B[0,r])}{r^{s-\alpha}}<+\infty,
\end{equation}
and such that, for any $x\in\R^n$ and any $0<r<\frac 12 |x|$, one has:
\begin{equation}\label{eq.cond-2}
\mu(B[x,r])\lesssim |x|^{-\alpha} r^s.
\end{equation}
\end{lemma}
\begin{proof}
Start by using Frostman's Lemma \cite[Theorem~8.8]{MATTILA} to find a non-zero (nonnegative) Radon measure $\nu$ supported in $B$ satisfying $\nu(B[x,r])\leq r^s$ for all $x\in\R^n$ and all $r>0$. Now define $A_k:=\{x\in\R^n: k\leq |x|<k+1\}$ for $k\in\N$ and introduce the Radon measure $\mu$ defined by:
$$
\mu:=\sum_{k=0}^\infty 2^{-k\alpha} \nu\hel A_k.
$$

Observe first that for $0<r<1$ one has:
$$
\frac{\mu(B[0,r])}{r^{s-\alpha}}=\frac{\nu(B[0,r])}{r^{s-\alpha}} \leq \frac{r^s}{r^{s-\alpha}}=r^\alpha\leq 1,
$$
while if one has $j\leq r<j+1$ for some $j\in\N^*$ there holds:
$$
\mu(B[0,r])\leq \sum_{k=0}^{j-1} 2^{-k\alpha} \nu(B[0,k+1])+2^{-j\alpha} \nu(B[0,r])\leq \sum_{k=0}^{j-1} 2^{-k\alpha} (k+1)^s+2^{-j\alpha} r^s,
$$
and hence:
\begin{multline*}
\frac{\mu(B[0,r])}{r^{s-\alpha}}\leq \frac{1}{r^{s-\alpha}}\sum_{k=0}^{j-1} 2^{-k\alpha} (k+1)^s+2^{-j\alpha} r^\alpha\\ \leq \sum_{k=0}^{j-1} 2^{-k\alpha} (k+1)^s +[2^{-j}(j+1)]^\alpha\leq C_{\alpha,s}<+\infty,
\end{multline*}
with for example $C_{\alpha,s}:=1+\sum_{k=0}^\infty 2^{-k\alpha} (k+1)^s$ since one has $(j+1)2^{-j}\leq 1$ for all $j\in\N^*$. This finishes to establish that \eqref{eq.cond-1} holds.\\

To prove \eqref{eq.cond-2}, fix $x\in\R^n$ and $0<r\leq\frac{|x|}{2}$. Choosing $j\in\N$ such that one has $j\leq |x|<j+1$, one finds $r\leq \frac{j+1}{2}$ and hence also, for $y\in B[x,r]$:
$$
|y|\geq |x|-|x-y|\geq j-\frac{j+1}{2}=\frac{j-1}{2}\quad\text{and}\quad |y|\leq |x|+|y-x|<j+1+\frac{j+1}{2}=\frac 32 (j+1),
$$
so that there holds $B[x,r]\cap A_k=\emptyset$ for $k<m_j:=\left\lfloor \frac{j-1}{2}\right\rfloor$ and $k>n_j:=\left\lceil \frac 32 (j+1)\right\rceil$.

We can hence compute:
\begin{equation}\label{eq.calc-1}
\mu(B[x,r])\leq \sum_{k=m_j}^{n_j} 2^{-k\alpha} \nu(B[x,r])\leq r^s \sum_{k=m_j}^{n_j} 2^{-k\alpha}.
\end{equation}
Yet one has:
\begin{multline}\label{eq.calc-2}
\sum_{k=m_j}^{n_j} 2^{-k\alpha}  = 2^{-m_j\alpha} \sum_{k=0}^{n_j-m_j} 2^{ -k\alpha}= 2^{-m_j\alpha}\frac{1-2^{-[1+(n_j-m_j)]\alpha}}{1-2^{-\alpha}}\\=\frac{2^{-m_j\alpha}-2^{-(n_j+1)\alpha}}{1-2^{-\alpha}}\leq \frac{1}{1-2^{-\alpha}} 2^{-m_j\alpha}\leq \frac{1}{1-2^{-\alpha}} 2^{-\left(\frac{j-1}{2}-1\right)\alpha)}=\frac{2^{\frac 32 \alpha}}{1-2^{-\alpha}} 2^{-\frac{j}{2}\alpha}.
\end{multline}
Writing then:
\begin{equation}\label{eq.calc-3}
 2^{-\frac{j}{2}\alpha} = |x|^{-\alpha} \left(\frac{|x|}{2^{\frac j2 }}\right)^\alpha\leq  |x|^{-\alpha} \left(\frac{j+1}{2^{\frac j2 }}\right)^\alpha
\leq \left(\frac 32 \right)^\alpha |x|^{-\alpha},
\end{equation}
since one has $\frac{k+1}{2^{\frac k2}}\leq \frac 32$ for any $k\in\N$, we finally get, combining \eqref{eq.calc-1}, \eqref{eq.calc-2} and \eqref{eq.calc-3}:
$$
\mu(B[x,r])\leq \frac{ 2^{\frac{\alpha}{2}}\cdot 3^\alpha}{1-2^{-\alpha}}\cdot |x|^{-\alpha} r^s,
$$
which establishes \eqref{eq.cond-2}.
\end{proof}

\section{Hausdorff dimension of removable sets for elliptic and canceling homogeneous operators}\label{sec.haus-dim}

Let $A(D)$ be a homogeneous linear differential operator on $\R^n$ of order $0<m<n$, from a finite-dimensional complex vector space $E$ to a finite dimensional complex vector space $F$, \emph{i.e.} an operator of the form:
$$
A(D)=\sum_{|\alpha|=m} c_\alpha \partial^\alpha : {\calD(\R^n,E)}\to {\calD(\R^n,F)},
$$
where $c_\alpha\in\calL(E,F)$ is a linear operator from $E$ to $F$, for each $|\alpha|=m$. Here, $\calD(\R^n,X)$ stands for the set of all smooth functions with compact support defined on $\R^n$ with values in a finite dimensional complex vector space $X$ (itself endowed with a fixed norm).

Recall that one associates to $A(D)$ its \emph{symbol} $A(\xi):E\to F$ defined by:
$$
A(\xi):=\sum_{|\alpha|=m} c_\alpha \xi^\alpha, \quad \xi \in \R^{n}.
$$
We then say that $A(D)$ is:
\begin{itemize}
\item[(i)] \emph{elliptic} in case its symbol $A(\xi)$ is injective for any $\xi\in\R^n\setminus \{0\}$;
\item[(ii)] \emph{canceling} in case one has:
\begin{equation}\label{canceling}
\bigcap_{\xi\in\R^n\setminus\{0\}} A(\xi)[E]=\{0\}.
\end{equation}
\end{itemize}
We also denote by $A^*(D):\calD(\R^n,F^*)\to \calD(\R^n,E^*)$ the formal adjoint of $A$.\\

An important class of operators satisfying (i) and (ii) is given by the gradient operator 
$A(D)=-\nabla$, where $E=\R$ {and} $F=\R^n$. Clearly the operator is elliptic, since its {symbol} is $A(\xi)=\xi$, and canceling for 
$n\geq 2$ since one has:
$$\bigcap_{\xi\in\R^n\setminus\{0\}} A(\xi)[E]= \bigcap_{\xi\in\R^n\setminus\{0\}} \xi \cdot \R = \{0\}.$$ 
{Note that $A^{*}(D)=\diver$. Then, for a positive (scalar) measure $\mu=|\mu|$ satisfying \eqref{eq.BP-1} and \eqref{eq.BP-2}, Theorem \ref{thm.BP} gives a solution $f \in L^{\infty}(\R^n,\R^n)$ for the equation \eqref{ad}.}\\  

We are now ready to prove our main Theorem~\ref{thm.main}.

\subsection{The proof of Theorem \eqref{thm.main}}

If the Hausdorff dimension of $S$ were larger than $n-m$, then there would exist $\alpha>0$ such that $\calH^{n-m+\alpha}(S)>0$. The above {Frostman Lemma} with power decay~---~Lemma~\ref{lem.frost}~---~applied to $B:=S$ and $s:=n-m+\alpha$ ensures the existence of a non-zero Radon measure supported in $S$ satisfying \eqref{eq.BP-1} 
and such that, for any $x\in\R^n$ and any $0<r<\frac 12 |x|$, one has:
$$
\mu(B[x,r])\lesssim |x|^{-\alpha} r^{n-m+\alpha}.
$$
Yet then, if $e\in E^*$ is fixed, and if one defines $\mu_e(B):=\mu(B) e$ for any $B\subseteq\R^n$, there holds:
$$
\sup_{r>0} \frac{|\mu_e|(B[0,r])}{r^{n-m}}=\|e\|_{E^\ast}\sup_{r>0} \frac{\mu(B[0,r])}{r^{n-m}}<+\infty,
$$
meaning that \eqref{eq.BP-1} above is fullfilled.

We also get, for any $x\in\R^n$, $x\neq 0$:
$$
\int_0^{\frac{|x|}{2}} \frac{|\mu_e|(B[x,r])}{r^{n-m+1}}\,dr=\int_0^{\frac{|x|}{2}} \frac{\mu(B[x,r])}{r^{n-m+1}}\,dr\lesssim |x|^{-\alpha}\int_0^{\frac{|x|}{2}} r^{\alpha-1}\,dr\leq \frac{1}{2^\alpha \alpha},
$$
so that \eqref{eq.BP-2} is also satisfied uniformly in $x\in\R^n$, $x\neq 0$.

Hence it follows from Theorem~\ref{thm.BP} that there exists $f\in L^\infty(\R^n,F^*)$ solving $A^*(D)f=\mu$~---~which implies that $S$ is \emph{not} removable for the equation $A^*(D) f=0$, since one has $A^*(D)f=0$ outside $S$ (in the sense of distributions) but $A^*(D)f=\mu\neq 0$ in $\R^n$ (in the sense of distributions).

\bibliographystyle{plain}
\bibliography{BMP}
\end{document}